\documentclass[11pt, reqno]{amsart}

\usepackage{fullpage,amsmath,amssymb,mathtools,enumerate,amsthm,comment}
\usepackage{makecell}
\usepackage[colorlinks, breaklinks]{hyperref}
\hypersetup{linkcolor=black,citecolor=blue,filecolor=black,urlcolor=black} 
\usepackage{tikz}
\usepackage{cleveref}
\usepackage{ textcomp }
\def\N{\mathbb{N}}

\def\p{\mathfrak{p}}
\def\q{\mathfrak{q}}
\def\mfp{\mathfrak{p}}
\def\mfm{\mathfrak{m}}

\def\mfq{\mathfrak{q}}

\DeclareMathOperator{\codepth}{\mathrm{codepth}}

\DeclareMathOperator{\height}{ht}
\DeclareMathOperator{\h}{ht}
\DeclareMathOperator{\ass}{Ass}
\DeclareMathOperator{\Ass}{Ass}
\DeclareMathOperator{\pdim}{pd}
\DeclareMathOperator{\pd}{pd}
\DeclareMathOperator{\reg}{reg}
\newcommand{\kk}{\mathbb{K}}					% The ground field.

\def\lra{\longrightarrow}

\newtheorem{thm}{Theorem}[section]
\newtheorem{lem}[thm]{Lemma}

\newtheorem{conj}[thm]{Conjecture}
\newtheorem{prop}[thm]{Proposition}
\newtheorem{qst}[thm]{Question}

\theoremstyle{definition}

\theoremstyle{remark}

\title{A Linear Bound on the Projective Dimension of Height 3 Quadratic Ideals}
\author{Zachary Greif}
\address{Iowa State University, Department of Mathematics, Ames, IA, USA}
\email{zsgreif@iastate.edu}

\author{Paolo Mantero}
\address{University of Arkansas, Department of Mathematical Sciences, Fayetteville, AR, USA}
\email{pmantero@uark.edu}

\author{Jason McCullough}
\address{Iowa State University, Department of Mathematics, Ames, IA, USA}
\email{jmccullo@iastate.edu}
\date{\today}



\begin{document}
%%%%%%%%%%%%%%%%%%%%%%%%%%%%%%%%%%%%%%%%%%%%%%%%%%%%%%%%%%%%%%%%%%%%%%%%%%%%%%%%%%%%%%%%%%%%%%%%%

\begin{abstract}
    In 2016, Ananyan and Hochster gave the first proof of a positive answer to Stillman's Question, which asked for a bound on the projective dimension of a graded polynomial ideal purely in terms of the number and degrees of its generators.  Explicit formulas for such a bound are limited and often not optimal.  In this paper, we give a nearly optimal linear upper bound on the projective dimension of height $3$ ideals generated by any number of degree $2$ homogenous polynomials.  
\end{abstract}
\maketitle

\section{Introduction}

Around 2000, Stillman posed the following question:

\begin{qst}[{Stillman's Question \cite[Problem 15.8]{PS09}}]
Does there exist a function $P : \N \times \N \to \N$ such that $\pdim_S(S/I) \leq P(n, d)$ holds for any polynomial ring $S$ over a field in some finite number of variables and any graded ideal $I\subseteq S$  generated by $n$ forms of degree at most $d$?
\end{qst}
\noindent That is, can we bound the projective dimension of $S/I$ in terms of the number of generators of $I$, and the degree of those generators, independent of the number of variables? The question is, in some sense, orthogonal to Hilbert's Syzygy theorem, which guarantees that the projective dimension is bounded by the number of variables of $S$.  For a detailed background on this problem, we refer the reader to \cite{MS13}, \cite{FMP16}, or \cite{ESS19}. 

The question was first positively answered by Ananyan and Hochster \cite{AH1}. Proofs employing different techniques were also given by Erman, Sam, and Snowden \cite{ESS} and Draisma, Lason, and Leykin \cite{DLL}. The proofs are either non-constructive, or rely on complicated recursive procedures, so these results raised the question of finding explicit upper bounds. 

A first observation is that any bounds must be quite large. Examples in \cite{Mc2} show that, when $I$ is generated by $3$ forms of degree $d$, there is a large lower bound: $P(3,d) \ge \left(\frac{d-1}{4}\right)^{(d-1)/2}+1$.
Recently, Caviglia, Liang, and Meng \cite{CLM25} gave bounds in the form of power-towers. While they are far from the conjectured optimal bounds they are the first explicit bounds for the general question. 
If one imposes some restrictions, better bounds are possible. For instance, in the quadratic case, Ananyan and Hochster \cite{AH2} proved that $P(n, 2) = 2^{n+1}(n-2) + 4$ suffices, but they conjectured that the general bound for $n$ quadrics should be quadratic in $n$. Specifically, they made the following conjecture; see also \cite[Question 6.2]{HMMS4}.
\begin{conj}[Ananyan and Hochster {\cite[Conjecture 11.4]{AH2}}]\label{quadboundconj}
If $S$ is a polynomial ring over a field and $I$ an ideal of $S$ generated by $n$ quadrics and of height $h$, then $\pd_S(S/I) \le h(n - h + 1)$.
\end{conj}
\noindent If true, this bound is necessarily tight, since the ideal
\[\left(x_1^2, x_2^2, \hdots, x_h^2, \sum_{k=1}^h a_{1,k}x_k, \hdots, \sum_{k=1}^h y_{n-h,k}x_k\right) \subseteq \kk[x_1,..,x_h,y_{1,1},\ldots,y_{n-h,k}]\]
has height $h$ and projective dimension exactly $h(n-h+1)$ if the $x_i$ and $a_{i,j}$ are all linearly independent elements of $S_1$; see  \cite[Theorem 3.3]{M1}.  Huneke et. al. \cite[Theorem 1.11]{HMMS1}, showed that Conjecture~\ref{quadboundconj} is true when $h \le 2$.
More generally, Ananyan and Hochster conjecture \cite[Conjecture 11.4]{AH2} that $P(n,d)$ can be taken of the form $C_d n^d$, where $C_d$ is a constant depending only on $d$.  If Conjecture~\ref{quadboundconj} is true, it would imply that $P(n,2) = \frac{(n+1)^2}{4}$ would suffice. 
We present a table of currently known nontrivial, explicit bounds in Table~\ref{table1}.

\begin{table}[htb]
\begin{center}
\begin{tabular}{|c|c|c|c|c|}
  \hline
  $d$ & $n$ & $h$ & Bound & Reference \\
  \hline
  $2$ & -- & -- & $2^{n+1}(n-2)+4$ & \cite[Theorem 1.11]{AH2} \\
  \hline
    $2$ & -- & 2 & $2n-2$ & \cite[Theorem 3.5]{HMMS1} \\
    \hline
  $2$ & 3 & -- & 4 & \cite[Theorem 1]{MS13} \\
  \hline
  $2$ & 4 & -- & 6 & \cite[Theorem 1.3]{HMMS3} \\
  \hline
  3 & 3 & -- & 5 & \cite[Main Theorem 1]{MM19}  \\
  \hline
  -- & -- & -- & $(d+1)$\textuparrow $(d^2+d+1)$\textuparrow $n$   & \cite[Theorem A]{CLM25} \\
  
  \hline
\end{tabular}
\end{center}
  \caption{Known explicit upper bounds for $\pd(S/I)$ for graded ideals $I \subseteq S$ of height $h$ generated by at most $n$ forms of degree at most $d$}\label{table1}
\end{table}
\noindent There are additional related results for cubics and quartics in \cite{AH2} which lead to very large values for $P(n,d)$.

The purpose of this paper is to give a close-to-optimal bound on the projective dimension of ideals generated by quadrics of height $3$.  Specifically, we show in Theorem~\ref{mainthm} that if $I = (f_1,\ldots,f_n)$ is a graded ideal of height $3$ generated by $n$ quadrics in a polynomial ring over a field, then \[\pd_S(S/I) \le \max\{3n+1,4n-10\},\]
which is close to the conjectured optimal bound of $3n - 6$ and much smaller than the exponential bound in \cite{AH2}.  Our methods are similar to those in \cite{HMMS1} and \cite{HMMS3} and we address why these methods likely can't be extended much further.  

The following section collects the necessary background results we need.  Section~\ref{SecBounds} contains the technical arguments needed for each possible multiplicity case.  The final Section~\ref{SecMain} contains the proof of the main result.

\section{Background}

Throughout, we let $\kk$ be a field, not necessarily algebraically closed, and $S$ be a polynomial ring over $\kk$ in some number of variables. We fix a graded ideal $I = (f_1, \hdots, f_n)$  of $S$, generated by quadrics unless specified otherwise. 

Any finitely generated graded $S$-module $M$ has a unique finite graded free resolution $\mathbf{F}_\bullet$ over $S$.  The ranks $\beta_{ij}$ of the graded free modules $F_i = \oplus_j S(-j)^{\beta_{ij}}$ are called the graded Betti numbers of $M$ and are conveniently displayed in its graded Betti table, where $\beta_{ij}$ is placed in column $i$ and row $j-i$.  One then easily reads off the projective dimension of $M$ as $\pd_S(M) = \max\{i\mid \beta_{ij} \neq 0\}$.  

We wish to give bounds for $\pdim_S(S/I)$ when $I$ is generated by quadrics and $\height(I) = 3$.  We will break this problem down based on multiplicity, $e(S/I)$. The associativity formula for multiplicity will allow us to reduce many cases to ones where $I$ has an associated prime of height 3 and a specified multiplicity. 

\begin{thm}[{Associativity Formula for Multiplicity cf. \cite[Theorem 14.7]{matsumura}}]
  If $I$ is an ideal of $S$, then
  \[e(S/I) = \sum_{\substack{\mfp \in \ass(I) \\ \height(\mfp) = \height(I)}} e(S/\mfp)\lambda(S_\mfp/I_\mfp).\]
\end{thm}

A graded prime ideal is nondegenerate if it does not contain any linear forms and thus defines a variety not contained in any hyperplane of the ambient projective space.  When $\kk$ is algebraically closed, there is a classical relationship between the height and multiplicity of nondegenerate prime ideals. 

\begin{prop} \label{multheightineq} (cf. \cite[Corollary 18.12]{harris})
  If $\kk$ is algebraically closed and $\mfp$ is a nondegenerate graded prime ideal, then \[e(S/\mfp) \geq \height(\mfp) + 1.\]
\end{prop}

\noindent We can apply this theorem iteratively to analyze and classify prime ideals of small height and multiplicity.  In particular, if $\mfp$ is a nondegenerate graded prime ideal with $e(S/\mfp) = \height(\mfp) + 1$, then $\mfp$ is said to be a prime ideal of \emph{minimal multiplicity}. If instead the equality $e(S/\mfp) = \height(\mfp) + 2$ is satisfied, $\mfp$ is said to be of \emph{almost minimal multiplicity}.

Prime ideals of minimal multiplicity (and their corresponding varieties) were completely classified by Bertini in 1907 \cite{bertini}, extending work on surfaces by Del Pezzo. A modern presentation of this classification is given by Eisenbud and Harris \cite{minimalmult}. 

\begin{thm}[{Classification of primes of minimal multiplicity, cf. \cite[Theorem 1]{minimalmult}}]\label{minmult}
  
  Suppose $\kk$ is algebraically closed.  If $\mfp \subseteq S$ is a nondegenerate graded prime ideal such that $e(S/\mfp) = \height(\mfp) + 1$, then $\mfp$ falls into one of three cases:

  \begin{enumerate}[(i)]
  \item \[\mfp = I_2\begin{pmatrix} x_0 & x_1 & x_2 \\ x_1 & x_3 & x_4 \\ x_2 & x_4 & x_5 \end{pmatrix},\]
    the ideal of $2 \times 2$ minors of a generic symmetric matrix,\\
    
  \item  \[\mfp = I_2\left(\begin{array}{cccc|ccc|cc}
                        x_{1,0} & x_{1,1} & \cdots & x_{1, a_0 - 1} & x_{1, 0} & \cdots & x_{1, a_1 - 1} & \cdots & x_{\ell, a_{\ell} - 1} \\
                        x_{1,1} & x_{1,2} & \cdots & x_{1, a_0} & x_{1, 1} & \cdots & x_{1, a_1} & \cdots & x_{\ell, a_{\ell}} \\
                      \end{array}\right),\]
                  the ideal of $2 \times 2$ minors of a $2 \times (\sum_i a_i)$ matrix for some positive integers $\ell, a_1, \hdots, a_\ell$, or\\
                  \item \[\mfp = (q),\]
                  that is, $\mfp$ is principly generated by an irreducible quadric.
   \end{enumerate}
\end{thm}
       
\noindent We note that case (i) defines (a cone over) the Veronese surface $\mathbb{P}^2 \subset \mathbb{P}^5$ and case (ii) defines (a cone over) various rational normal scrolls.  In particular, case (i) has codimension $3$, while the codimension in case (ii) is $(\sum_i a_i) - 1$.

Brodmann and Schenzel  \cite{BS06} gave a classification of nondegenerate primes of almost minimal multiplicity  when $\height(\mfp) \leq 4$. In particular, when $\height(\mfp) = 3$, this classification gives us three possible Betti tables for $S/\mfp$. 

\begin{thm}[{Classification of primes of almost minimal multiplicity - height 3 case \cite[Theorem 2.2]{BS06}}]\label{AMM-bound}  Suppose $\kk$ is algebraically closed. 
  If $\mfp \subseteq S$ is a nondegenerate graded prime ideal with $\height(\mfp) = 3$ and $e(S/\mfp) = 5$ which is not a cone, and let $X = V(\mfp)$ be the associated projective variety.  One of the following holds:
  \begin{enumerate}[(i)]
      \item \vspace{.3cm}$\dim X \le 4$ and $\mfp$ is generated by the five $4 \times 4$ Pfaffians of a skew-symmetric $5 \times 5$ matrix of linear forms. Its minimal free resolution is given by the Buchsbaum–Eisenbud complex with the following Betti table.\\
         \[ \begin{tabular}{r|ccccc}
      &$0$&$1$&$2$&$3$\\
      \hline
      $0:$&$1$&\text{-}&\text{-}&\text{-}\\$1:$&\text{-}&$5$&$5$&\text{-}\\$2:$&\text{-}&\text{-}&\text{-}&$1$\\
      \end{tabular}\]
      
      \item \vspace{.3cm} $\dim X \le 4$ and $\codepth(S/\mfp)  = 1$. The Betti table $S/\mfp$ has the following form.\\
       \[ \begin{tabular}{r|cccccc}
      &$0$&$1$&$2$&$3$&$4$\\
      \hline
      $0:$&$1$&\text{-}&\text{-}&\text{-}&\text{-}\\$1:$&\text{-}&$4$&$2$&\text{-}&\text{-}\\$2:$&\text{-}&$1$&$6$&$5$&$1$\\
      \end{tabular}\]
      
    \item \vspace{.3cm} $\dim X \le 5$ and $\codepth(S/\mfp) = 2$. The Betti diagram of $S/\mfp$ has the following form.\\
      \[ \begin{tabular}{r|ccccccc}
      &$0$&$1$&$2$&$3$&$4$&$5$\\
      \hline
      $0:$&$1$&\text{-}&\text{-}&\text{-}&\text{-}&\text{-}\\$1:$&\text{-}&$3$&$2$&\text{-}&\text{-}&\text{-}\\$2:$&\text{-}&$6$&$16$&$15$&$6$&$1$\\
            \end{tabular}\]
  \end{enumerate}
\end{thm}
\noindent Note that $\dim X = \dim(S/\mfp) - 1$; here $\mathrm{codepth}(S/\mfp) = \dim(S/\mfp) - \mathrm{depth}(S/\mfp) = \pd_S(S/I) - \height(I)$.

We will also need the following upper bound on the multiplicity of certain graded almost complete intersections.

\begin{thm}[{\cite[Theorem 2.2]{HMMS4}}]\label{maxmult}
Let $J$ be a graded Cohen-Macaulay $S$-ideal, and let $F \notin J$ be homogeneous. Set $I = J + (F)$ and assume $\height I = \height J$. Then
\begin{enumerate}[(i)]
\item $e(R/I) \le e(R/J) - \max\{1,s(R/J) - 
\deg(F) + 1\}$;
\item if equality is achieved in (i) and $\deg(F) \le s(R/J)$, then $R/I$ is Cohen-Macaulay.
\end{enumerate}
\end{thm}

\noindent In the previous result, $s(R/J) = \reg(R/J) - \pd(R/J)$ denotes the socle degree of an Artinian reduction. 

Next, we recall that if the generators of an ideal are contained in a polynomial subring generated (as an algebra) by a regular sequence, one gets effective bounds on projective dimension.  The following proposition is essentially Hilbert's Syzygy Theorem together with flatness.

\begin{prop}[{\cite[Lemma 3.3]{AH0}}] \label{regseqsubalgebra}
    Let $I = (f_1, \hdots, f_m) \subseteq S$ be a homogeneous ideal. Suppose we have a homogeneous regular sequence $g_1, \hdots, g_t$ such that $f_i \in \kk[g_1, \hdots, g_t]$ for each $i$. Then $\pdim_S(S/I) \leq t$.
\end{prop}
\noindent In particular, if the generators of $I$ can be expressed in terms of at most $t$ linear forms, than $\pd_S(S/I) \le t$.

Finally, we recall the notion of residual intersections, a generalization of liaison, introduced by Artin and Nagata \cite{AN72}.  An ideal $J$ is called an $s$-residual intersection of an ideal $I$ if there is an $s$-generated ideal $A \subseteq I$ with $J = A : I$ and $\height(J) \ge s \ge \height(I)$. If, in addition, $\height(I + J) > \height(J)$, then $J$ is said to be a geometric $s$-residual intersection of $I$.  We refer the reader to \cite{Huneke83} and \cite{HU88} for further results on residual intersections.

\section{Bounds Based on Multiplicity of Associated Primes}\label{SecBounds}

In this section, we will obtain bounds on $\pdim_S(S/I)$ based on when $S/I$ has a minimal prime $\mfp$ of height $3$ with a given multiplicity.   We tackle each case in a separate lemma.
In particular, we will prove bounds for the cases where $\height(\mfp) = 3$, and $e(S/\mfp) = 1,\ldots,5$. 
Throughout this section we will require that $\kk$ is algebraically closed and $I = (f_1,\ldots,f_n) \subseteq S$ is a graded ideal generated by $n$ quadrics. 

For an arbitrary ideal $I$ generated by $n$ quadrics which has an associated prime $\mfp$ of multiplicity $1$, one easily sets that $\pd_S(S/I) \le h(n+1)$, where $h = \height(\mfp)$. As a first step, we slightly improve this bound when $\height(\mfp)=3$.

\begin{lem} \label{mult1prime}
 If $I$ has an associated prime $\mfp$ with $\height(\mfp) = 3$ and $e(S/\mfp) = 1$, then $$\pdim_S(S/I) \leq 3n+1.$$
\end{lem}

\begin{proof}
  Since $e(S/\mfp) = 1 < \height(\mfp)+1$, $\mfp$ must be degenerate by Proposition~\ref{multheightineq}. Applying this fact iteratively, we must have that $\mfp = (\ell, \ell', \ell'')$, where $\ell, \ell', \ell'' \in S_1$ are linearly independent. Then each $f_i$ can be written as $f_i = a_i\ell + b_i\ell' + c_i \ell''$, where all $a_i, b_i$, and $c_i$ are in $S_1$. Thus, we need at most $3n + 3$ variables to write the $f_i$, so $\pdim_S(S/I) \leq 3n+3$ by Proposition~\ref{regseqsubalgebra}. 
  %\end{proof}

  Suppose that $\ell,\ell',\ell'',a_1,\ldots,a_{n},b_1,\ldots,b_{n},c_1,\ldots,c_{n}$ are linearly independent.  Then $J = I:\mfp$ is a generic $n$-residual intersection.  By \cite[Theorem 3.3]{HU88}, $J$ is a geometric $n$-residual intersection, $J$ is prime, $S/J$ is Cohen-Macaulay, and $\height(J) = n$.  
  Let
  \[\mathbf{M} = \begin{pmatrix} a_1 & a_2 & \cdots & a_{n}\\
  b_1 & b_2 & \cdots & b_{n} \\
  c_1 & c_2 & \cdots & c_{n}\\
  \end{pmatrix}.\]
  Let $D = I_3(\mathbf{M})$ denote the ideal of $3 \times 3$ minors of $\mathbf{M}$, which is prime and $\height(D) = n-2$.  By Cramer's Rule, we have $D \subseteq I:\mfp$. Thus $\height(I + \mfp) \ge \height(D + \mfp) = n+1$.  Now by \cite[Corollary 2.10]{HMMS1}, $\pd_S(S/I) = n$.  
  
  Therefore we may suppose that  $\ell,\ell',\ell'',a_1,\ldots,a_{n},b_1,\ldots,b_{n},c_1,\ldots,c_{n}$ are linearly dependent.  If their span has dimension at most $3n+1$, we are done.  Else, we can assume that the linear forms $\ell,\ell',\ell'',a_1,\ldots,a_{n-1},b_1,\ldots,b_{n-1},c_1,\ldots,c_{n-1}$ are linearly independent.  Let $I' = (f_1,\ldots,f_{n-1})$ so that $I = I' + (f_n)$ and set $J' = I':\mfp$.  As above, $J'$ is a generic $(n-1)$-residual intersection.  By \cite[Theorem 3.1]{Huneke83}, $I' = J' \cap \mfp$.  Since $f_n \in \mfp \smallsetminus J'$, we get $I':f_n = J'$.  As above $\pd_S(S/J') = \pd_S(S/I') = n-1$.  It follows from the short exact sequence
  \[0 \to S/J' \to S/I' \to S/I \to 0\]
  that $\pd_S(S/I) \le n$.  This completes the proof.
\end{proof}

The following lemma is well-known. For instance, it can be easily obtained as a consequence of \cite[Theorem 2.5]{HNTT20}.

\begin{lem}\label{lemma:adding:linear:forms} Let $S$ be a polynomial ring over a field $\kk$, and let $S' = S[\underline{z}]$, where $\underline{z} = z_1,\ldots,z_u$.  Let $H'$ be any graded $S'$-ideal containing $z_1,\ldots,z_u$, and write it as $H'=H+(\underline{z})$, where $H \subseteq S$ is an ideal.
Then every associated prime $\p'$ of $H'$ has the form $\p'=\p S'+(\underline{z})$ for some $\p\in \Ass_{S}(S/H)$, and  $\height(\p')=\height(\p)+ u$.
\end{lem}

The next result, which may be of independent interest, describes the primary decomposition of any ideal generated by quadrics which is contained in a linear prime ideal of height $2$. We will employ it to tackle the case where  $I$ has an associated prime $\mfp$ with $\height(\mfp) = 3$ and $e(S/\mfp) = 2$.

\begin{thm}\label{primary}
Let $J=(q_1,\ldots,q_t)\subseteq S$ be generated by quadrics and contained in a linear prime $(x,y)$ of height 2.  Write $q_i = a_ix + b_iy$ for $1 \le i \le t$, and set $\mathfrak{m} = (x,y,a_1,\ldots,a_t,b_1,\ldots,b_t)$.
For any $\mfp \in \ass(S/J)$ with $\mfp\neq \mathfrak{m}$, one has  
\[
\height(\mfp)\leq t.
\]
\end{thm}

\begin{proof}

The ideal $J$ is extended from an ideal $J'\subseteq S':=\kk[x,y,a_1,\ldots,a_t,b_1,\ldots,b_t]$. 
Then, by \Cref{lemma:adding:linear:forms}, after possibly replacing $S$ with $S'$ we may assume $\mfm$ is the graded maximal ideal of $S$.  In particular, $\mfm$ is fixed by any linear change of variables.

If $\height(J)=1$, then we can write $J=\ell J'$ for some linear form $\ell\in (x,y)$ and some ideal $J'$ generated by $t$ linear forms, so $\height(\mfp)\leq \pd(S/J)\leq t$.

We may then assume $\h(J)=2$. Since $\p\in \Ass(S/J)$,  by the Auslander-Buchsbaum formula
$$
\h(\p)=\dim(S_\p)=\pd(S_\p/J_\p),
$$
so we will prove (the equivalent statement) $\pd(S_\p/J_\p)\leq t$. We proceed by induction on $t\geq 2$. If $t=2$, since $\h(J)=2$,  $J$ is a complete intersection, and the statement is trivially true. Assume now $t>2$. Combining  \Cref{lemma:adding:linear:forms} and induction, we obtain the following fact which we will use later in the proof.\\
\\
{\bf Fact.} {\em If $H= (q_1,\ldots,q_{t-u}) + (z_1,\ldots,z_u)$, where $u>0$ and $\underline{z}=z_1,\ldots,z_u$ are linear forms with $\height(\underline{z})\geq 1$, and $q_1,\ldots,q_{t-u}$ are quadrics contained in a height 2 linear prime, then $\pd(S_\p/H_\p)\leq t$.}\\
\\
After possibly taking linear combinations of the generators, and possibly changing the $a$'s and $b$'s, one may write\\

\begin{tabular}{ll}
& $J = (a_1x_1+b_1y, \ldots, a_rx+ b_ry,\; x(a_{r+1},\ldots,a_s) ,\;y(b_{s+1},\ldots,b_t)),$  \\
$(\star) \qquad{}$ & for some $r\leq s\leq t$, where $a_1,\ldots,a_r,y$ and $b_1,\ldots,b_r,x$ are both regular sequences,  
\\
& as are the sequences $a_1,\ldots,a_s$ and $b_1,\ldots,b_r,b_{s+1},\ldots,b_t$.
\end{tabular}\medskip

Now, if the matrix of the coefficients $
\left[\begin{matrix}
a_1 & a_2 & \ldots & a_t\\
b_1 & b_2 & \ldots & b_t
\end{matrix}\right]$ 
is 1-generic, then \cite[Proposition~5.1]{HMMS1} yields $\pd(S/J)=t$, yielding in particular $\pd(S_\p/J_\p)\leq t$. 

We may then assume the above matrix contains at least one generalized zero so, without loss of generality, after potentially taking linear combinations of the generators, we may assume $s-r\geq 1$ in $(\star)$.

If $x\notin \p$, then $J_\p = H_\p$, where $H=(a_1x_1+b_1y, \ldots, a_rx+ b_ry) + y(b_{s+1},\ldots,b_t) + (a_{r+1},\ldots,a_s)$ and, since ${\rm depth}(S_\p/H_\p)={\rm depth}(S_\p/J_\p)=0$, we have $\p\in \Ass(S/H)$.  
Then, by the above Fact, $\pd(S_\p/H_\p)\leq t$. 
Similarly, we are done if $t-s\geq 1$ and $y\notin \p$.

We may then assume $x\in \p$. Observe that if $(a_{r+1},\ldots,a_s)\not\subseteq \p$, then $x\in J_\p$, so we can write $J_\p=(x,y(b_1,\ldots,b_r,b_{s+1},\ldots,b_t))_\p$ and so $\pd(T_\p/J_\p)\leq 1 + r+(t-s)\leq t$. We may then assume $(a_{r+1},\ldots,a_s)\subseteq \p$. Symmetrically, we may assume $(b_{s+1},\ldots,b_t)\subseteq \p$.

It then suffices to prove the statement under the following assumptions:
\begin{center}
$J$ has form $(\star)$ with $s-r\geq 1$ and $(x,a_{r+1},\ldots,a_s,b_{s+1},\ldots,b_t)\subseteq \p$.
\end{center}
\noindent \underline{\textbf{Case 1: $y\notin \p$.}} Since above we proved the inequality when $t-s\geq 1$ and $y\notin \p$, we may then assume $s=t$. We can write 
\[
J_\p=(b_1',\ldots,b_r',x'(a_{r+1},\ldots,a_s)),
\]
where $x' = x/y$ and $b_j'=a_jx' + b_j$ for $1\leq j \leq r$.
Observe that $(b_1',\ldots,b_r',x') _{\p} =(b_1,\ldots,b_r,x)_{\p}$. Since $b_1,\ldots,b_r,x$ is a regular sequence in $\p\subseteq S$, then $b_1',\ldots,b_r',x'$ is a regular sequence in $\p S_\p$. We then obtain that $J_\p: x' = (b_1,',\ldots,b_r', a_{r+1},\ldots,a_s)$. The short exact sequence
\[
0 \lra S_\p/(b_1',\ldots,b_r', a_{r+1},\ldots,a_s)_{\p} \stackrel{\cdot x'}{\lra} S_\p/J_\p \lra  S_\p/(x',y)_{\p} \lra 0
\]
yields $\pd(S_\p/J_\p)\leq \pd(S_\p/(b_1',\ldots,b_r', a_{r+1},\ldots,a_s)_{\p})$. We observe that $(b_1,',\ldots,b_r', a_{r+1},\ldots,a_s)_{\p} = H_\p$, where $H$ is the $S$-ideal generated by the following $r$ quadrics and $s-r\geq 1$ linearly independent linear forms
\[
H=(a_1x+b_1y, \ldots, a_rx+b_ry) + (a_{r+1},\ldots,a_s),
\]
so we are done by the above Fact. 
\noindent \underline{\textbf{Case 2: $y\in \p$.}} We now have $(x,y,a_{r+1},\ldots,a_s,b_{s+1},\ldots,b_t)\subseteq \p\neq \mfm$, so, after possibly relabelling, either $a_1\notin \p$ or $b_1\notin \p$. \\

\noindent \underline{\textbf{Case 2.i: $b_1 \not\in \p$.}} Set $a_1'=a_1/b_1$, $a_i' = a_i - b_ia_1'$ for $2\leq i \leq r$, and $y'=y+a_1'x$. One can check that
\[ 
J_\p=(y',\; x(a_2',\ldots,a_r', a_{r+1},\ldots,a_s, \; a_1'(b_{s+1},\ldots,b_t)))_{\p}.
\]
 Since $(x,y)_\p=(x,y')_{\p}$, we have $x,y'$ is a regular sequence in $S_\p$; thus, we have the short exact sequence 
\[
0 \lra S_\p/(y',a_2',\ldots,a_r', a_{r+1},\ldots,a_s, \; a_1'(b_{s+1},\ldots,b_t))_{\p} \stackrel{\cdot x}{\lra} S_\p/J_\p \lra  S_\p/(x,y')_{\p} \lra 0.
\]
It then suffices to show the projective dimension of the left-hand side is at most $t$. %$\pd(S_\p/(y',a_2',\ldots,a_r', a_{r+1},\ldots,a_s, \; a_1'(b_{s+1},\ldots,b_t)))\leq t$. 
To prove this, observe that $(y',a_2',\ldots,a_r', a_{r+1},\ldots,a_s, \; a_1'(b_{s+1},\ldots,b_t))_{\p} = H_\p$, where $H$ is the following $S$-ideal generated by $t-(s-r)$ quadrics and $s-r\geq 1$ linearly independent linear forms:
\[
H=(a_1x+b_1y, a_2b_1-a_1b_2,\ldots, a_rb_1-a_1b_r, a_1(b_{s+1},\ldots,b_t)) + (a_{r+1},\ldots,a_s). %,a_{r+1},\ldots,a_s).
\]
Then,  $\pd(S_\p/J_\p)\leq \pd(S_\p/H_\p)\leq t$, where the leftmost inequality holds from the above, and the rightmost by the above Fact.\\
\\
\noindent \underline{\textbf{Case 2.ii: $a_1 \not\in \p$.}} Set $b_1'=b_1/a_1$, $b_i' = b_i - a_ib_1'$ for $2\leq i \leq r$ and $x'=x+b_1'y$. Then, 
\[ 
J_\p=(x',\; y(b_2',\ldots,b_r', b_1'(a_{r+1},\ldots,a_s), b_{s+1},\ldots,b_t))_{\p}.
\]
 Since $(x,y)_\p=(x',y)_{\p}$, the elements $x',y$ form a regular sequence in $\p S_\p$, thus, we have the short exact sequence 
\[
0 \lra S_\p/(x', b_2',\ldots,b_r', b_1'(a_{r+1},\ldots,a_s), b_{s+1},\ldots,b_t)_{\p} \stackrel{\cdot y}{\lra} S_\p/J_\p \lra  S_\p/(x',y)_{\p} \lra 0,
\]
so it suffices to show $\pd(S_\p/J')\leq t$, where $J'=(x', b_2',\ldots,b_r', b_1'(a_{r+1},\ldots,a_s), b_{s+1},\ldots,b_t)_{\p} \subseteq S_\p$. 

If $t-s \geq 1$, we are done as in case 2.i.
If instead $t=s$, then in the previous argument we cannot apply the inductive hypothesis, so we proceed as follows. As $t=s$, the ideal is $J'=(x', b_2',\ldots,b_r',\; b_1'(a_{r+1},\ldots,a_t))_{\p} \subseteq S_\p$. Observe that $(b_1',b_2',\ldots,b_r',x')_{\p} =(b_1,\ldots,b_r,x)_\p$ is a complete intersection in $\p S_\p$. Thus, we have the following 
short exact sequence 
\[
0 \lra S_\p/(x', b_2',\ldots,b_r', a_{r+1},\ldots,a_t)_{\p} \stackrel{\cdot y}{\lra} S_\p/J'\lra   S_\p/(b_1',\ldots,b_r',x')_{\p} \lra 0,
\]
and it suffices to prove $\pd(S_\p/(x', b_2',\ldots,b_r', a_{r+1},\ldots,a_t)_{\p})\leq t$. Like the above argument, we observe that $(x', b_2',\ldots,b_r', a_{r+1},\ldots,a_t)_{\p}=H_\p$, where $H$ is the following $S$-ideal generated  by $r<t$ quadrics and $t-r$ linearly independent linear forms:
\[
H = (a_1x+b_1y, a_1b_2 -a_2b_1,\ldots, a_1b_r - a_rb_1) + (a_{r+1},\ldots,a_t).
\]
Then,  $\pd(S_\p/J_\p)\leq \pd(S_\p/J')\leq \pd(S_\p/H_\p)\leq t$, where the rightmost inequality follows by the above Fact.
\end{proof}

\begin{lem} \label{mult2prime}
 If $I$ has an associated prime $\mfp$ with $\height(\mfp) = 3$ and $e(S/\mfp) = 2$, then $$\pdim_S(S/I) \leq \max\{3n-4, 4n-10\}.$$
\end{lem}

\begin{proof}
  If $\mfp$ is nondegenerate, then we would have $2 = e(S/\mfp) \geq \height(\mfp) + 1 = 4$, so $\mfp$ must be degenerate by Proposition~\ref{multheightineq}. In particular, this means $\mfp = (\ell) + \mfp'$, where $\ell \in S_1$, $\mfp'$ is prime with $\height(\mfp) = 2$ and $e(S/\mfp') = 2$. Similarly, $\mfp'$ must also be degenerate, otherwise $e(S/\mfp') \geq 1 + \height(\mfp') = 3$. So,   $\mfp = (\ell, \ell', q)$ for some $q \in S_2$. Since $\ell,\ell'$ are linearly independent linear forms in $S$, we may write $S=S'[\ell,\ell']$ and assume $q\in S'$. After subtracting a multiple of $q$ from each $f_i$, without loss of generality we can assume $f_n = q$ and write $f_i = a_i\ell + b_i \ell'$ for $1 \le i \le n-1$, where each $a_i, b_i \in S'_1$.

As $J = (f_1,\ldots,f_{n-1})$ is an ideal of height $2$ generated by quadrics, $\pd_S(S/J) \le 2n - 4$ by \cite[Theorem 3.5]{HMMS1}.
If $f_n$ is a nonzerodivisor on $S/J$, then $\pd_S(S/J) \le 2n - 3$. If not, then $f_n$ is contained in some associated prime $\mfq$ of $S/J$.

Let $T=\kk[\ell,\ell',a_1,\ldots,a_{n-1},b_1,\ldots,b_{n-1}]$, and let $\mathfrak{m}$ be its maximal ideal.
Since polynomial extensions preserve primary decompositions \cite[Exercise 4.7]{Atiyah-Macdonald}, every associated prime of $S/J$ is extended from an associated prime of $T/(J \cap T)$. If $\mathfrak{m}\notin {\rm Ass}(T/(J \cap T))$, then, by \Cref{primary}, $\height(\mfq)\leq n-1$ holds.   If $\ell, \ell', a_1,\ldots,a_{n-1},b_1,\ldots,b_{n-1}$ are linearly independent, then the primary decomposition of $J$ is
$J = (\ell, \ell') \cap I_2(\mathbf{M})$, where $\mathbf{M}$ is a generic $2 \times (n-1)$ matrix of linear forms.  Since $f_n\notin (\ell,\ell')$, it follows that $f_n \in I_2(\mathbf{M})$, and so $\pd(S/I) \le 2n - 2$, because $f_n\in \kk[a_1,\ldots,a_{n-1},b_1,\ldots,b_{n-1}]$.

So, we may assume that $\height(\ell,\ell',a_1,\ldots,a_{n-1},b_1,\ldots,b_{n-1}) \le 2n-1$.
Since $\height(\mfq) \le n-1$,  $\mfq$ contains at most $n-1$ linearly independent linear forms.  
We write $f_n=F_0 + F_1$ for some quadrics $F_1\in \q$ and $F_0$ in $T$. So, $F_0$ is written purely in terms of $\ell,\ell'$ and at most $2n-3$ other linearly independent linear forms in $T$. However, since $f_n\in S'$, we may further assume $F_0$ is written using at most $2n-3$ linear forms from $T$. Since to write $F_1$ it takes at most $n-1$ new linear forms, in addition to the (at most) $2n-3$ linear forms of $T$, then it takes at most $3n-4$ linear forms to express $f_n$. It follows from Proposition~\ref{regseqsubalgebra} that $\pd_S(S/I) \le 3n-4$. 

We may then assume $\mathfrak{m}\in {\rm Ass}(T/J)$, (and $f_n$ is regular on every other element of ${\rm Ass}(T/J)$). Since $\height(\mathfrak{m})=\mathrm{bight}(J) \le \pd_S(S/J) \le 2n-4$, we have $\dim(T)\leq 2n-4$. Since $f_n \in S'$ we may assume $f_n$ lies in a polynomial ring with at most $(2n-4)-2=2n-6$ variables. It then takes at most $2n-6$ new linear forms to express $f_n$.  This, in addition to the at most $2n-4$ variables generating $T$, shows that  $\pd_S(S/I) \le (2n-4) + (2n-6) = 4n - 10$.

Now if $n = 3$, $I$ is a complete intersection and $\pd(S/I) = 3 \le 3\cdot 3 - 4$.  If $n = 4$, then $\pd(S/I) \le 6 \le 3 \cdot 4 - 4$ by \cite{HMMS3}.  If $n \ge 5$, then $\pd(S/I) \le \max\{2n, 3n-4, 4n-10\} \le \max\{3n-4, 4n-10\}$.  
\end{proof}

The previous multiplicity $2$ case is the only case where our bound exceeds the conjectured bound by more than a constant. The difficulty in making up the difference stems from the difficulty in specifying a primary decomposition for the ideal $J$ in the previous proof. 

\begin{lem} \label{mult3prime}
    If $I$ has an associated prime $\mfp$ with $\height(\mfp) = 3$ and $e(S/\mfp) =  3$, then $$\pdim_S(S/I) \leq n+7.$$
\end{lem}

\begin{proof}
    Since $e(S/\mfp) < \height(\mfp) + 1$, Proposition~\ref{multheightineq} implies that $\mfp$ is degenerate. Thus $\mfp = (\ell) + \mfp'$, where $\height{\mfp'} = 2$ and $e(S/\mfp') = 3$. If $\mfp'$ is also degenerate, then $\mfp = (\ell, \ell', h)$, where $h \in S_3$. Then all quadrics $f_i \in (\ell, \ell')$, contradicting $\height(I) = 3$.
    
    So we can assume $\mfp'$ is a height $2$  nondegenerate prime ideal. By Theorem~\ref{minmult}, $\mfp'$ is generated by the $2\times 2$ minors, $\Delta_1, \Delta_2, \Delta_3$ of a $2 \times 3$ matrix with at most 6 distinct linear forms as entries. Then $f_i = a_i \ell + \sum_{j=1}^3 \alpha_{ij} \Delta_j$, where $a_i \in S_1$ and $\alpha_{ij} \in \kk$ for each $i$. So we need at most $n + 7$ variables ($n$ for the $a_i$, 1 for $\ell$, and 6 to write the $\Delta_i$) to express all of the generators of $I$. Thus by Proposition~\ref{regseqsubalgebra}, $\pdim_S(S/I) \leq n+7$.
\end{proof}

\begin{lem} \label{mult4prime}
 If $I$ has an associated prime $\mfp$ with $\height(\mfp) = 3$ and $e(S/\mfp) =  4$, then $$\pdim_S(S/I) \leq \max\{n,8\}.$$
\end{lem}

\begin{proof}
If $\mfp$ is nondegenerate, then $\mfp$ is a prime of minimal multiplicity, in which case $\mfp$ is generated by the $2 \times 2$ minors of a $2 \times 4$ matrix of linear forms (the scroll case) or the $2 \times 2$ minors of a $3 \times 3$ symmetric matrix of linear forms. Thus it takes at most $8$ variables to express the quadrics in $\mfp$, and hence $\pd_S(S/I) \le 8$.  

Otherwise, $\mfp$ is degenerate.  To contain a height 3 ideal of quadrics, $\mfp$ must contain a regular sequence $\ell, q, q'$, where $\ell \in S_1$ and $q, q' \in S_2$.  Since $S/(\ell,q,q')$ is Cohen-Macaulay, it is unmixed of multiplicity $4$.  By \cite[Lemma 8]{engheta2}, $\mfp = (\ell,q,q')$.
After taking a linear combination of the generators, we can assume that $f_n = q$ and $f_{n-1} = q',$ so that $(x, q, q') = (x, f_{n-1},f_n)$. Then we can also assume that for $1 \leq i \leq n-2$, $f_i = a_i x$ for some linearly independent $a_i \in S_1$.

Consider the short exact sequence
\[0 \to S/(I:x) \to S/I \to S/(I + (x)) \to 0.\]
Clearly $I + (x) = (x,q,q')$ and $\pd_S(S/(I+(x)) = 3$.  Since $x,q,q'$ is a regular sequence, we also see that $I:x = (a_1,\ldots,a_{n-2},q,q')$, which is an ideal generated by at most $n-2$ linearly independent linear forms and two quadrics.  Thus $\pd_S(S/(I:x)) \le n$.  By the short exact sequence above, we have $\pd_S(S/I) \le n$ as well.
\end{proof}

\begin{lem} \label{mult5prime}
     If $I$ has an associated prime $\mfp$ with $\height(\mfp) = 3$ and $e(S/\mfp) = 5$, then $$\pdim_S(S/I) \leq 6.$$
\end{lem}

\begin{proof}
    If $\mfp$ is nondegenerate, then $\mfp$ is one of the primes of almost minimal multiplicity in Theorem~\ref{AMM-bound}.  If $n \le 4$, then $\pd_S(S/I) \le 6$ by \cite[Theorem 1.3]{HMMS1}, so we may assume $n \ge 5$.  If $n = 5$, then the only option is that $I = \mfp$ and $\mfp$ is the ideal of Pfaffians from Case~(i) of Theorem~\ref{AMM-bound}.  Thus $\pd_S(S/I) = 3$.  If $n \ge 6$, then we get a contradiction as $\mfp$ contains at most $5$ linearly independent quadrics by Theorem~\ref{AMM-bound}.

    Otherwise $\mfp$ is degenerate, so $\mfp = (\ell) + \mfp'$, where $\ell \in S_1$ and $\mfp'$ is a prime with $\height(\mfp') = 2$ and $e(S/\mfp') = 5$. If $\mfp'$ is also degenerate, then $\mfp = (\ell, \ell', h)$, where $\ell' \in S_1$ and $h \in S_5$. In this case $I \subseteq (\ell, \ell')$, which contradicts $\height(I) = 3$.
    
    So  $\mfp'$ must be nondegenerate. Let $J = (\mfp'_2)$, the ideal generated by the quadrics in $\mfp'$.   If $\height(J) = 1$, then $I \subseteq (\ell) + J$, so $\height(I) \leq \height((\ell) + J) = 2$, a contradiction.
    So $\height(J) = 2$, in which case there exists a regular sequence $q, q'$ of quadrics in $J$.  As $\mfp' \supseteq J \supseteq (q,q')$ are ideals of height $2$, we have $e(S/\mfp') \leq  e(S/J) \leq e(S/(q, q')) = 4.$ This contradicts $e(S/\mfp) = e(S/\mfp') = 5$.

\end{proof}

\section{The Main Result}\label{SecMain}

We collect the individual bounds from the previous section to prove our main result below.

\begin{thm}\label{mainthm} Let $\kk$ be a field and $S$ a polynomial ring over $\kk$.
  If $I = (f_1, \hdots, f_n) \subseteq S$ is a graded ideal of height 3 generated by $n$ quadrics, then 
  \[\pdim_S(S/I) \leq \max\{3n+1,4n-10\}.\]
\end{thm}

\begin{proof}  Note that we can assume $\overline{\kk} = \kk$ without changing $\pd_S(S/I)$.
	Since $\height (I) = 3$ we can choose the $f_i$ such that $f_1, f_2, f_3$ is a regular sequence of length 3. 
	We will break into cases based on the multiplicity of $S/I$, first noting that $e(S/I) = 8$ if and only if $I$ is a complete intersection, in which case $\pdim_S(S/I) = 3$. 
	
	If $I$ is not a complete intersection, then $n \ge 4$.  Applying Theorem~\ref{maxmult}(i) with $J = (f_1,f_2,f_3)$, we get that $e(S/I) \le e(S/(f_1,f_2,f_3,f_4)) \le 6$.  If $e(S/I) = 6$, then $e(S/(f_1,f_2,f_3,f_4)) = 6$.  By Theorem~\ref{maxmult}(ii), $S/(f_1,f_2,f_3,f_4)$ is Cohen-Macaulay, and in particular, unmixed.  Since $I \supseteq (f_1,f_2,f_3,f_4)$ and since both ideals have height $3$ and multiplicity $6$, they must be equal \cite[Lemma 8]{engheta2}.  Therefore $\pd_S(S/I) = \height I = 3$.
	
	So we may assume $1 \leq e(S/I) \leq 5$.  By the associativity formula, $I$ must have an associated prime of height $3$ and multiplicity $1, 2, 3, 4,$ or $5$. Thus, to bound $\pdim_S(S/I)$, it is sufficient to take the maximum of the bounds in Lemmas~\ref{mult1prime}, \ref{mult2prime}, \ref{mult3prime}, \ref{mult4prime}, and \ref{mult5prime}.
\end{proof}

Note that for $n \ge 11$, the maximum achieved by $4n-10$.  Partial progress has been made analyzing the projective dimension of height $4$ almost complete intersections of quadrics.  One can use similar methods to bound the projective dimension when the multiplicity is small; see \cite{EK19}.  However, without stronger classification theorems for prime ideals or more general bounds, it seems difficult to push these methods much further.

\section*{Acknowledgements}

 Computations with Macaulay2 \cite{M2} were very helpful in the writing of this paper. Mantero was partially supported by Simons Foundation Grant \#962192. McCullough was supported by  National Science Foundation grants DMS--1900792 and DMS--2401256. 

\bibliographystyle{alpha}
\bibliography{h3qi}

@article {AH0,
    AUTHOR = {Ananyan, Tigran and Hochster, Melvin},
     TITLE = {Ideals generated by quadratic polynomials},
   JOURNAL = {Math. Res. Lett.},
  FJOURNAL = {Mathematical Research Letters},
    VOLUME = {19},
      YEAR = {2012},
    NUMBER = {1},
     PAGES = {233--244},
      ISSN = {1073-2780},
   MRCLASS = {13D05 (13A15 13F20)},
  MRNUMBER = {2923188},
MRREVIEWER = {Marco Fontana},
       DOI = {10.4310/MRL.2012.v19.n1.a18},
       URL = {https://doi.org/10.4310/MRL.2012.v19.n1.a18},
}

@article {AH1,
    AUTHOR = {Ananyan, Tigran and Hochster, Melvin},
     TITLE = {Small subalgebras of polynomial rings and {S}tillman's
              {C}onjecture},
   JOURNAL = {J. Amer. Math. Soc.},
  FJOURNAL = {Journal of the American Mathematical Society},
    VOLUME = {33},
      YEAR = {2020},
    NUMBER = {1},
     PAGES = {291--309},
      ISSN = {0894-0347},
   MRCLASS = {13D05 (13F20)},
  MRNUMBER = {4066476},
       DOI = {10.1090/jams/932},
       URL = {https://doi.org/10.1090/jams/932},
}

@article {AH2,
    AUTHOR = {Ananyan, Tigran and Hochster, Melvin},
     TITLE = {Strength conditions, small subalgebras, and {S}tillman bounds
              in degree {$\leq 4$}},
   JOURNAL = {Trans. Amer. Math. Soc.},
  FJOURNAL = {Transactions of the American Mathematical Society},
    VOLUME = {373},
      YEAR = {2020},
    NUMBER = {7},
     PAGES = {4757--4806},
      ISSN = {0002-9947},
   MRCLASS = {13D05 (13F20)},
  MRNUMBER = {4127862},
       DOI = {10.1090/tran/8060},
       URL = {https://doi.org/10.1090/tran/8060},
}

@article {AN72,
    AUTHOR = {Artin, M. and Nagata, M.},
     TITLE = {Residual intersections in {C}ohen-{M}acaulay rings},
   JOURNAL = {J. Math. Kyoto Univ.},
  FJOURNAL = {Journal of Mathematics of Kyoto University},
    VOLUME = {12},
      YEAR = {1972},
     PAGES = {307--323},
      ISSN = {0023-608X},
   MRCLASS = {13H10},
  MRNUMBER = {301006},
MRREVIEWER = {Tadayuki\ Matsuoka},
       DOI = {10.1215/kjm/1250523522},
       URL = {https://doi.org/10.1215/kjm/1250523522},
}

@misc{CLM25,
      title={Explicit {S}tillman bounds for all degrees}, 
      author={Giulio Caviglia and Yihui Liang and Cheng Meng},
      year={2025},
      eprint={2507.19617},
      archivePrefix={arXiv},
      primaryClass={math.AC},
      url={https://arxiv.org/abs/2507.19617}, 
}

@article {PS09,
    AUTHOR = {Peeva, Irena and Stillman, Mike},
     TITLE = {Open problems on syzygies and {H}ilbert functions},
   JOURNAL = {J. Commut. Algebra},
  FJOURNAL = {Journal of Commutative Algebra},
    VOLUME = {1},
      YEAR = {2009},
    NUMBER = {1},
     PAGES = {159--195},
      ISSN = {1939-0807},
   MRCLASS = {13D02 (13D40)},
  MRNUMBER = {2462384},
MRREVIEWER = {Elena Guardo},
       DOI = {10.1216/JCA-2009-1-1-159},
       URL = {https://doi.org/10.1216/JCA-2009-1-1-159},
}

@article {ESS,
    AUTHOR = {Erman, Daniel and Sam, Steven V. and Snowden, Andrew},
     TITLE = {Big polynomial rings and {S}tillman's conjecture},
   JOURNAL = {Invent. Math.},
  FJOURNAL = {Inventiones Mathematicae},
    VOLUME = {218},
      YEAR = {2019},
    NUMBER = {2},
     PAGES = {413--439},
      ISSN = {0020-9910},
   MRCLASS = {13D02 (13A02)},
  MRNUMBER = {4011703},
MRREVIEWER = {Haohao Wang},
       DOI = {10.1007/s00222-019-00889-y},
       URL = {https://doi.org/10.1007/s00222-019-00889-y},
}

@article {DLL,
    AUTHOR = {Draisma, Jan and Laso\'{n}, Micha\l  and Leykin, Anton},
     TITLE = {Stillman's conjecture via generic initial ideals},
   JOURNAL = {Comm. Algebra},
  FJOURNAL = {Communications in Algebra},
    VOLUME = {47},
      YEAR = {2019},
    NUMBER = {6},
     PAGES = {2384--2395},
      ISSN = {0092-7872},
   MRCLASS = {13A02 (13E99 13F25 13P10)},
  MRNUMBER = {3957104},
MRREVIEWER = {Irena Swanson},
       DOI = {10.1080/00927872.2019.1574806},
       URL = {https://doi.org/10.1080/00927872.2019.1574806},
}

@book {Atiyah-Macdonald,
    AUTHOR = {Atiyah, M. F. and Macdonald, I. G.},
     TITLE = {Introduction to commutative algebra},
 PUBLISHER = {Addison-Wesley Publishing Co., Reading, Mass.-London-Don
              Mills, Ont.},
      YEAR = {1969},
     PAGES = {ix+128},
   MRCLASS = {13.00},
  MRNUMBER = {0242802},
MRREVIEWER = {J. A. Johnson},
}

@book {harris,
    AUTHOR = {Harris, Joe},
     TITLE = {Algebraic geometry},
    SERIES = {Graduate Texts in Mathematics},
    VOLUME = {133},
      NOTE = {A first course},
 PUBLISHER = {Springer-Verlag, New York},
      YEAR = {1992},
     PAGES = {xx+328},
      ISBN = {0-387-97716-3},
   MRCLASS = {14-01},
  MRNUMBER = {1182558},
MRREVIEWER = {Liam O'Carroll},
       DOI = {10.1007/978-1-4757-2189-8},
       URL = {https://doi.org/10.1007/978-1-4757-2189-8},
}

@book {bertini,
AUTHOR = {Bertini, Eugenio},
TITLE = {Introduzione alla geometria proiettive degli iperspazi},
PUBLISHER = {Enrico Spoerri, Pisa},
YEAR = {1907},
}

@article {BS06,
    AUTHOR = {Brodmann, Markus and Schenzel, Peter},
     TITLE = {On varieties of almost minimal degree in small codimension},
   JOURNAL = {J. Algebra},
  FJOURNAL = {Journal of Algebra},
    VOLUME = {305},
      YEAR = {2006},
    NUMBER = {2},
     PAGES = {789--801},
      ISSN = {0021-8693},
   MRCLASS = {14M07 (14J10)},
  MRNUMBER = {2266853},
MRREVIEWER = {Maria Luisa Spreafico},
       DOI = {10.1016/j.jalgebra.2006.03.027},
       URL = {https://doi.org/10.1016/j.jalgebra.2006.03.027},
}

@incollection {minimalmult,
    AUTHOR = {Eisenbud, David and Harris, Joe},
     TITLE = {On varieties of minimal degree (a centennial account)},
 BOOKTITLE = {Algebraic geometry, {B}owdoin, 1985 ({B}runswick, {M}aine,
              1985)},
    SERIES = {Proc. Sympos. Pure Math.},
    VOLUME = {46},
     PAGES = {3--13},
 PUBLISHER = {Amer. Math. Soc., Providence, RI},
      YEAR = {1987},
   MRCLASS = {14J40 (14J26)},
  MRNUMBER = {927946},
MRREVIEWER = {Allen B. Altman},
       DOI = {10.1090/pspum/046.1/927946},
       URL = {https://doi.org/10.1090/pspum/046.1/927946},
}

@article {ESS19,
    AUTHOR = {Erman, Daniel and Sam, Steven V. and Snowden, Andrew},
     TITLE = {Cubics in 10 variables vs. cubics in 1000 variables:
              uniformity phenomena for bounded degree polynomials},
   JOURNAL = {Bull. Amer. Math. Soc. (N.S.)},
  FJOURNAL = {American Mathematical Society. Bulletin. New Series},
    VOLUME = {56},
      YEAR = {2019},
    NUMBER = {1},
     PAGES = {87--114},
      ISSN = {0273-0979},
   MRCLASS = {13A02 (13D02)},
  MRNUMBER = {3886145},
MRREVIEWER = {Haohao Wang},
       DOI = {10.1090/bull/1652},
       URL = {https://doi.org/10.1090/bull/1652},
}

@article {FMP16,
    AUTHOR = {Fl{\o}ystad, Gunnar and McCullough, Jason and Peeva, Irena},
     TITLE = {Three themes of syzygies},
   JOURNAL = {Bull. Amer. Math. Soc. (N.S.)},
  FJOURNAL = {American Mathematical Society. Bulletin. New Series},
    VOLUME = {53},
      YEAR = {2016},
    NUMBER = {3},
     PAGES = {415--435},
      ISSN = {0273-0979},
   MRCLASS = {13D02},
  MRNUMBER = {3501795},
MRREVIEWER = {Liam O'Carroll},
       DOI = {10.1090/bull/1533},
       URL = {https://doi.org/10.1090/bull/1533},
}

@Misc{M2,
          author = {Grayson, Daniel R. and Stillman, Michael E.},
          title = {Macaulay2, a software system for research in algebraic geometry},
          howpublished = {Available at \url{http://www.math.uiuc.edu/Macaulay2/}}
        }

@article {HNTT20,
    AUTHOR = {H\`a, Huy T\`ai and Nguyen, Hop Dang and Trung, Ngo Viet and
              Trung, Tran Nam},
     TITLE = {Symbolic powers of sums of ideals},
   JOURNAL = {Math. Z.},
  FJOURNAL = {Mathematische Zeitschrift},
    VOLUME = {294},
      YEAR = {2020},
    NUMBER = {3-4},
     PAGES = {1499--1520},
      ISSN = {0025-5874,1432-1823},
   MRCLASS = {13C15 (13D07 18G15)},
  MRNUMBER = {4074049},
MRREVIEWER = {Elo\'isa\ Grifo},
       DOI = {10.1007/s00209-019-02323-8},
       URL = {https://doi.org/10.1007/s00209-019-02323-8},
}

@article {Huneke83,
    AUTHOR = {Huneke, Craig},
     TITLE = {Strongly {C}ohen-{M}acaulay schemes and residual
              intersections},
   JOURNAL = {Trans. Amer. Math. Soc.},
  FJOURNAL = {Transactions of the American Mathematical Society},
    VOLUME = {277},
      YEAR = {1983},
    NUMBER = {2},
     PAGES = {739--763},
      ISSN = {0002-9947},
   MRCLASS = {13H10 (14M05)},
  MRNUMBER = {694386},
MRREVIEWER = {Ulrich Orbanz},
       DOI = {10.2307/1999234},
       URL = {https://doi.org/10.2307/1999234},
}

@article {HMMS1,
    AUTHOR = {Huneke, Craig and Mantero, Paolo and McCullough, Jason and
              Seceleanu, Alexandra},
     TITLE = {The projective dimension of codimension two algebras presented
              by quadrics},
   JOURNAL = {J. Algebra},
  FJOURNAL = {Journal of Algebra},
    VOLUME = {393},
      YEAR = {2013},
     PAGES = {170--186},
      ISSN = {0021-8693},
   MRCLASS = {13D05 (13D02 14C40)},
  MRNUMBER = {3090065},
MRREVIEWER = {Haohao Wang},
       DOI = {10.1016/j.jalgebra.2013.06.038},
       URL = {https://doi.org/10.1016/j.jalgebra.2013.06.038},
}

@article {HMMS4,
    AUTHOR = {Huneke, Craig and Mantero, Paolo and McCullough, Jason and
              Seceleanu, Alexandra},
     TITLE = {A multiplicity bound for graded rings and a criterion for the
              {C}ohen-{M}acaulay property},
   JOURNAL = {Proc. Amer. Math. Soc.},
  FJOURNAL = {Proceedings of the American Mathematical Society},
    VOLUME = {143},
      YEAR = {2015},
    NUMBER = {6},
     PAGES = {2365--2377},
      ISSN = {0002-9939},
   MRCLASS = {13C14 (13D40 13H15)},
  MRNUMBER = {3326019},
MRREVIEWER = {Florian Enescu},
       DOI = {10.1090/S0002-9939-2015-12612-3},
       URL = {https://doi.org/10.1090/S0002-9939-2015-12612-3},
}

@article {HMMS3,
    AUTHOR = {Huneke, Craig and Mantero, Paolo and McCullough, Jason and
              Seceleanu, Alexandra},
     TITLE = {A tight bound on the projective dimension of four quadrics},
   JOURNAL = {J. Pure Appl. Algebra},
  FJOURNAL = {Journal of Pure and Applied Algebra},
    VOLUME = {222},
      YEAR = {2018},
    NUMBER = {9},
     PAGES = {2524--2551},
      ISSN = {0022-4049},
   MRCLASS = {13D05 (13C40 13D02 14M07)},
  MRNUMBER = {3783004},
MRREVIEWER = {Keri Sather-Wagstaff},
       DOI = {10.1016/j.jpaa.2017.10.005},
       URL = {https://doi.org/10.1016/j.jpaa.2017.10.005},
}

@article {HU88,
    AUTHOR = {Huneke, Craig and Ulrich, Bernd},
     TITLE = {Residual intersections},
   JOURNAL = {J. Reine Angew. Math.},
  FJOURNAL = {Journal f\"{u}r die Reine und Angewandte Mathematik. [Crelle's
              Journal]},
    VOLUME = {390},
      YEAR = {1988},
     PAGES = {1--20},
      ISSN = {0075-4102},
   MRCLASS = {13H10 (13C13 13D10 14B07)},
  MRNUMBER = {953673},
MRREVIEWER = {J\"{u}rgen Herzog},
       DOI = {10.1515/crll.1988.390.1},
       URL = {https://doi.org/10.1515/crll.1988.390.1},
}

@book {matsumura,
    AUTHOR = {Matsumura, Hideyuki},
     TITLE = {Commutative ring theory},
    SERIES = {Cambridge Studies in Advanced Mathematics},
    VOLUME = {8},
      NOTE = {Translated from the Japanese by M. Reid},
 PUBLISHER = {Cambridge University Press, Cambridge},
      YEAR = {1986},
     PAGES = {xiv+320},
      ISBN = {0-521-25916-9},
   MRCLASS = {13-01},
  MRNUMBER = {879273},
MRREVIEWER = {W. V. Vasconcelos},
}

@article {M1,
    AUTHOR = {McCullough, Jason},
     TITLE = {A family of ideals with few generators in low degree and large
              projective dimension},
   JOURNAL = {Proc. Amer. Math. Soc.},
  FJOURNAL = {Proceedings of the American Mathematical Society},
    VOLUME = {139},
      YEAR = {2011},
    NUMBER = {6},
     PAGES = {2017--2023},
      ISSN = {0002-9939},
   MRCLASS = {13D05 (13D02)},
  MRNUMBER = {2775379},
MRREVIEWER = {Liana M. \c{S}ega},
       DOI = {10.1090/S0002-9939-2010-10792-X},
       URL = {https://doi.org/10.1090/S0002-9939-2010-10792-X},
}

@article {Mc2,
    AUTHOR = {McCullough, Jason},
     TITLE = {Prime ideals and three-generated ideals with large regularity},
   JOURNAL = {C. R. Math. Acad. Sci. Paris},
  FJOURNAL = {Comptes Rendus Math\'ematique. Acad\'emie des Sciences. Paris},
    VOLUME = {362},
      YEAR = {2024},
     PAGES = {251--255},
      ISSN = {1631-073X,1778-3569},
   MRCLASS = {13D02 (13D05 13P20)},
  MRNUMBER = {4745334},
MRREVIEWER = {Dipankar\ Ghosh},
       DOI = {10.5802/crmath.544},
       URL = {https://doi.org/10.5802/crmath.544},
}

@article {EK19,
    AUTHOR = {El Khoury, Sabine},
     TITLE = {On the projective dimension of 5 quadric almost complete
              intersections with low multiplicities},
   JOURNAL = {Rocky Mountain J. Math.},
  FJOURNAL = {The Rocky Mountain Journal of Mathematics},
    VOLUME = {49},
      YEAR = {2019},
    NUMBER = {5},
     PAGES = {1491--1546},
      ISSN = {0035-7596},
   MRCLASS = {13D02 (13C40 13D05)},
  MRNUMBER = {4010571},
MRREVIEWER = {Quang Hoa Tran},
       DOI = {10.1216/rmj-2019-49-5-1491},
       URL = {https://doi.org/10.1216/rmj-2019-49-5-1491},
}

@article {engheta2,
    AUTHOR = {Engheta, Bahman},
     TITLE = {On the projective dimension and the unmixed part of three
              cubics},
   JOURNAL = {J. Algebra},
  FJOURNAL = {Journal of Algebra},
    VOLUME = {316},
      YEAR = {2007},
    NUMBER = {2},
     PAGES = {715--734},
      ISSN = {0021-8693},
   MRCLASS = {13F20 (13C40 16E10)},
  MRNUMBER = {2358611},
MRREVIEWER = {Wenting Tong},
       DOI = {10.1016/j.jalgebra.2006.11.018},
       URL = {https://doi.org/10.1016/j.jalgebra.2006.11.018},
}

@article {MM19,
    AUTHOR = {Mantero, Paolo and McCullough, Jason},
     TITLE = {The projective dimension of three cubics is at most 5},
   JOURNAL = {J. Pure Appl. Algebra},
  FJOURNAL = {Journal of Pure and Applied Algebra},
    VOLUME = {223},
      YEAR = {2019},
    NUMBER = {4},
     PAGES = {1383--1410},
      ISSN = {0022-4049},
   MRCLASS = {13D05 (13D02 14M06 14M07)},
  MRNUMBER = {3906508},
MRREVIEWER = {Haohao Wang},
       DOI = {10.1016/j.jpaa.2018.06.009},
       URL = {https://doi.org/10.1016/j.jpaa.2018.06.009},
}

@incollection {MS13,
    AUTHOR = {McCullough, Jason and Seceleanu, Alexandra},
     TITLE = {Bounding projective dimension},
 BOOKTITLE = {Commutative algebra},
     PAGES = {551--576},
 PUBLISHER = {Springer, New York},
      YEAR = {2013},
   MRCLASS = {13D05 (13P20 16E10)},
  MRNUMBER = {3051385},
MRREVIEWER = {Ellen E. Kirkman},
       DOI = {10.1007/978-1-4614-5292-8\_17},
       URL = {https://doi.org/10.1007/978-1-4614-5292-8_17},
}

\end{document}